\documentclass[a4paper,10pt,leqno, english]{amsart}
\title[Completion  Theorem]{The Completion  Theorem  in  Twisted  Equivariant  K-Theory for  proper  actions.}
\author{No\'e B\'arcenas }
\address{Centro de Ciencias Matem\'aticas. UNAM \\ Ap.Postal 61-3 Xangari. Morelia, Michoac\'an M\'EXICO 58089}
\email{barcenas@matmor.unam.mx}
\urladdr{http://www.matmor.unam.mx/~barcenas}

\author{Mario Vel\'asquez }
\address{Departamento de Matem\'aticas. \\Pontificia Universidad Javeriana\\Cra. 7 No. 43-82 - Edificio Carlos Ort\'iz 5to piso\\ Bogot\'a D.C, Colombia}

 \email{mavelasquezm@gmail.com}
 \urladdr{https://sites.google.com/site/mavelasquezm/}

\usepackage{hyperref}
\usepackage{pdfsync}
\usepackage{calc}
\usepackage{enumerate,amssymb}
\usepackage[arrow,curve,matrix,tips,2cell]{xy}
\usepackage{comment}
  \SelectTips{eu}{10} \UseTips
  \UseAllTwocells

\DeclareMathAlphabet\EuR{U}{eur}{m}{n}
\SetMathAlphabet\EuR{bold}{U}{eur}{b}{n}

\makeindex             % used for the subject index
\theoremstyle{plain}
\newtheorem{theorem}{Theorem}[section]
\newtheorem{lemma}[theorem]{Lemma}
\newtheorem{proposition}[theorem]{Proposition}
\newtheorem{corollary}[theorem]{Corollary}

\newtheorem*{theoremn}{Theorem}
\theoremstyle{definition}
\newtheorem{definition}[theorem]{Definition}

\newtheorem{remark}[theorem]{Remark}

{\catcode`@=11\global\let\c@equation=\c@theorem}

\newcommand{\comsquare}[8]                   % Produces a commutative square
{\begin{CD}
#1 @>#2>> #3\\
@V{#4}VV @V{#5}VV\\
#6 @>#7>> #8
\end{CD}
}

\newcommand{\xycomsquare}[8]                   % kommutatives Quadrat (xy-Version)
{\xymatrix
{#1 \ar[r]^{#2} \ar[d]^{#4} &
#3 \ar[d]^{#5}  \\
#6\ar[r]^{#7} &
#8
}
}

\newcommand{\calfin}{\mathcal{FIN}}

\newcommand{\calh}{\mathcal{H}}

\newcommand{\calu}{{\mathcal U}}

\newcommand{\IC}{{\mathbb C}}

\newcommand{\IK}{{\mathbb K}}

\newcommand{\IP}{{\mathbb P}}

\newcommand{\IZ}{{\mathbb Z}}

\newcommand{\Hg}{\mathcal{\widehat{H}}}

\newcommand{\curs}{\EuR}

\newcommand{\Or}{\curs{Or}}

\newcommand{\charac}{\operatorname{char}}

\newcommand{\colim}{\operatorname{colim}}

\newcommand{\Ext}{\operatorname{Ext}}
\newcommand{\Hom}{\operatorname{Hom}}

\newcommand{\im}{\operatorname{im}}

\newcommand{\ind}{\operatorname{ind}}

\newcommand{\res}{\operatorname{res}}

\newcommand{\pt}{\{\bullet\}}

\newcommand{\bredoncover}[3]{H^{#1}_{\mathbb{Z} \cover}(#3, #2)}
\newcommand{\bredoncell}[3]{H^{#1}_{\mathbb{Z} \OrGF{G}{\mathcal{FIN}}}(#3, #2)}
\newcommand{\cover}{\mathcal{N}_{G} \mathcal{U}}

\newcommand{\ktheory}[3]{K_{#1}^{#2}(#3, P)}
\newcommand{\Fred}{\ensuremath{{\mathrm{Fred}}}}

\newcommand{\UU}{\mathcal{U}}
\newcommand{\HH}{\mathcal{H}}
\newcommand{\BB}{\mathcal{B}}
\newcommand{\KK}{\mathcal{K}}

\newcommand{\idealGX}[2]{{\bf I}_{#1,#2}}

\newcommand{\eub}[1]{\underline{E}#1}              %Eunderbar G = classifying space for proper G-actions
\newcommand{\OrGF}[2]{\Or_{#2}(#1)}                %orbit category
\newcommand{\higherlim}[3]{{\setbox1=\hbox{\rm lim}
        \setbox2=\hbox to \wd1{\leftarrowfill} \ht2=0pt \dp2=-1pt
        \mathop{\vtop{\baselineskip=5pt\box1\box2}}
        _{#1}}^{#2}#3}

\newcommand{\version}[1]                       %marks the date of last editing and compilation
{\begin{center} last edited on #1\\
last compiled on \today\
name of texfile: \jobname
\end{center}
}

\newcounter{commentcounter}

\begin{document}

\typeout{----------------------------  linluesau.tex  ----------------------------}

%%%%%%%%%%%%%%%%%%%%%%%%%%%%%%%%%%%%%%%%%%%%%%%%%%%%%%%%%%%%%%%%%%%%%%%%%%%%%%%%%
%%%%%%%%%%%%%%%%%%%%%%%%%%%%%%%%%%% Abstract  %%%%%%%%%%%%%%%%%%%%%%%%%%%%%%%%%%%%%%%
%%%%%%%%%%%%%%%%%%%%%%%%%%%%%%%%%%%%%%%%%%%%%%%%%%%%%%%%%%%%%%%%%%%%%%%%%%%%%%%%%

%%%%%%%%%%%%%%%%%%%%%%%%%%%%%%%%%%%%%%%%%%%%%%%%%%%%%%%%%%%%%%%%%%%{March 18th, 2011%%%%%%%%%%%%%%
%%%%%%%%%%%%%%%%%%%%%%%%%%%%%%%%%% Introduction %%%%%%%%%%%%%%%%%%%%%%%%%%%%%%%%%%%%%
%%%%%%%%%%%%%%%%%%%%%%%%%%%%%%%%%%%%%%%%%%%%%%%%%%%%%%%%%%%%%%%%%%%%%%%%%%%%%%%%\section{Introduction}
% \typeout{-------------------------------   Section 0: Introduction --------------------------------}
\setcounter{section}{0}

%\begin{comment}

\begin{abstract}
We  compare  different  algebraic  structures  in  twisted  equivariant  $K$-Theory  for  proper  actions  of  discrete  groups. After the  construction  of a  module  structure  over  untwisted  equivariant  K-Theory,  we  prove  a  completion  Theorem  of  Atiyah-Segal  type for  twisted  equivariant  K-Theory. Using  a  Universal  Coefficient  Theorem,  we  prove  a  cocompletion  Theorem  for  Twisted Borel K-Homology for  discrete  Groups.
\end{abstract}

 \maketitle

The  Completion  Theorem in  equivariant  $K$-theory by  Atiyah and  Segal \cite{atiyahsegalcompletion} had a remarkable influence  on  the  development  of topological $K$-theory and  computational methods related  to  it.

Twisted  equivariant  $K$-theory  for  proper  actions  of  discrete  groups was  defined  in \cite{barcenasespinozajoachimuribe}  and  further  computational  tools,  notably  a  version  of  Segal's  spectral  sequence,   have  been  developed  by  the  authors  and collaborators in \cite{barcenasespinozauribevelasquez} and \cite{barcenasvelasquez}. 

In this  work,   we  examine  twisted  equivariant  $K$-theory with  the  above mentioned  methods  as  a  module  over  its  untwisted  version and  prove  a  generalization  of  the completion  theorem by  Atiyah and  Segal.

It  turns  out  that  in the  case   of  groups  which  admit a  finite  model  for  the  classifying   space  for  proper  actions $\eub{G}$, the ring   defined  as  the  zeroth (untwisted) $G$-equivariant  $K$-theory  ring  $K_G^0(\eub{G})$ is   Noetherian. Hence,  usual  commutative  algebraic  methods  can  be  applied  to   deal  with  completion problems on  noetherian  modules over  it,  as  it  has  been  done  in  other  contexts  in  the  literature, \cite{atiyahsegalcompletion}, \cite{lueckoliverbundles},  \cite{cantarerocompletion}, \cite{lahtinenatiyahsegal}.  

Using  a  Universal  Coefficient  Theorem  developed  in the  analytical  setting \cite{rosenbergschochet}, we  prove  a  version  of  the  co-completion  theorem   in twisted Borel  Equivariant K-homology,  thus  extending  results   in \cite{joachimlueck} to the  twisted  case. 

This  work  is  organized  as  follows: 

In  section \ref{sectioncompletion}, we  collect  results  on the   multiplicative (twist-mixing) structures on  twisted  equivariant  $K$-theory  following  its  definition  in \cite{barcenasespinozajoachimuribe}. 
 We  also recall  in this  section  the  spectral  sequence  of  \cite{barcenasespinozauribevelasquez} and  the  necessary  notions  of  Bredon-type cohomology and  $G$-CW  complexes.
 
 In  section \ref{sectionmodule}, we  examine the  ring  structure  over  the ring defined  by  the  zeroth untwisted $K$-theory   $K_G^0(\eub{G})$,  and  establish  the   noetherian  condition  for certain  relevant modules  over  it  given by  twisted  equivariant $K$-theory  groups. 

The main theorem,  \ref{theoremcompletion} is  proved  in  section \ref{sectionatiyahsegal}. 

%\commentm{Definir el ideal de aumentacion, no lo encontré en ningun lugar}

\begin{theoremn}
Let  $G$  be a  group  which  admits  a  finite  model  for $\eub{G}$, the  universal  space  for  proper  actions. Let  $X$ be  a  finite, proper $G$-CW  complex.  Let  $\idealGX{G}{\eub{G}}$ be  the  augmentation ideal (defined in section  \ref{sectionatiyahsegal} ).

Let $p:X\times EG\to X$ be denote the projection to the first coordinate. For  any projective unitary, $G$-equivariant  stable  bundle  $P$,     the  pro-homomorphism 
$$\varphi_{\lambda,p}: \big \{ K_G^*(X,P)/ {{\bf I}_{G,\eub{G}}}^n K_G^*(X,P)  \big \}\longrightarrow  \big \{K_G^*(X\times EG^{n-1},p^{*}(P))\big \} $$
is  a  pro-isomorphism. In  particular, the  system $\big \{K_G^*(X\times EG^{n-1},p^{*}(P))\big \}$ satisfies  the  Mittag-Leffler  condition  and  the  $\lim^ 1$ term is  zero. 
\end{theoremn}

  Finally,  section \ref{sectioncocompletion} deals  with  the  proof  of  the  cocompletion  theorem \ref{theoremcocompletion} involving Twisted Borel   $K$-homology.

  \begin{theoremn}
Let $G$  be  a  discrete  group. 

Assume  that  $G$  admits  a  finite  model  for  $\eub{G}$, and  let $X$ be a   finite  $G$-CW complex.  Let  $\idealGX{G}{\eub{G}}$ be  the  augmentation ideal (defined in section  \ref{sectioncompletion}).

For  any projective unitary, $G$-equivariant  stable  bundle  $P$, there  exists a  short  exact  sequence  
\begin{multline*}0\to \colim_{n\geq1} \Ext^{1}_\mathbb{Z} (K^*_G(X,P)/\idealGX{G}{\eub{G}}^n\cdot K_G^*(\eub{G},P) , \mathbb{Z})   \to \\ K_*(X\times_G EG, p^*(P)) \to \colim_{n\geq 1} \Hom_\IZ(K^*_G(X,P)/\idealGX{G}{\eub{G}}^n\cdot K_G^*(\eub{G},P) ,\IZ)\to0
\end{multline*}
%\commentm{El ultimo termino es el dual?}
\end{theoremn}

\tableofcontents

\subsection*{Aknowledgments}
The first  author  thanks  the  support  of PAPIIT  research grant IA100315.

 The  second author  thanks  partial support  of  a  UNAM Postdoctoral Fellowship, as  well as  partial support  by the project \emph{Index morphism in twisted K-theory} with ID 8165 from Faculty of Sciences of Pontificia Universidad Javeriana, Bogot\'a, Colombia.

The  first  and  second  author thank  Universit\'e  Toulouse  III Paul Sabatier, as  well  as  the  Laboratoire International  Solomon  Lefschetz (LAISLA) for  support during  a  visit to  Toulouse,  where  parts  of  this  work  were  written.

\newpage

\section{Preliminaries  on (Twisted) Equivariant  K-theory  for  Proper  and  Discrete  Actions}\label{sectioncompletion}

\begin{definition}Let $G$ be a discrete group.
Recall  that  a  $G$-CW  complex structure  on  the  pair $(X,A)$ of topological $G$-spaces consists  of a  filtration of  the $G$-space $X=\cup_{-1\leq n } X_{n}$ with $X_{-1}=A$, where  every  space   is inductively  obtained  from  the  previous  one   by  attaching  $G$-cells  in $G$-equivariant pushout  diagrams  
$$\xymatrix{\coprod_{i} S^{n-1}\times G/H_{i} \ar[r] \ar[d] & X_{n-1} \ar[d] \\ \coprod_{i}D^{n}\times G/H_{i} \ar[r]& X_{n}}$$

 Recall  that  a  $G$-CW-complex  is  proper  if  the stabilizer  subgroups  of  points  are all  finite.  We  say  that a  proper  $G$-CW complex  is  finite  if  it  is  constructed  out    of  a  finite  number  of  cells of  the  form  $G/H\times  D^n$. 

\end{definition}

We  recall the  notion  of  the  classifying  space  for  proper  actions following \cite{lueckclassifying}: 
\begin{definition}
 Let  $G$ be  a  discrete group.  A model  for  the classifying  space  for  proper  actions is  a $G$-CW  complex  $\eub{G}$  with  the  following  properties: 
\begin{itemize}
 \item{All  isotropy  groups  are  finite.}
\item{For  any  proper  $G$-CW  complex $X$  there  exists a  unique $G$-map $X\to \eub{G}$ up  to $G$-homotopy  . }
\end{itemize}
\end{definition}

The  classifying  space  for  proper  actions  always  exists, it is  unique  up  to $G$-homotopy  equivalence. The  following  list   contains  some  examples. We  refer to \cite{lueckclassifying} for  further  discussion.

 \begin{itemize}
\item{If  $G$  is  a  finite group,  then  the  singleton  space  is  a  model  for  $\eub{G}$. }
\item{Let $G$ be  a  group  acting properly  and  cocompactly  on  a     ${\rm Cat}(0)$ space  $X$. Then  $X$ is a model for  $\eub{G}$. }
\item{Let  $G$  be  a  Coxeter  group. The  Davis complex  is  a model for   $\eub{G} $. }See \cite{daviscomplex}.
\item{Let  $G$ be  a   mapping  class  group  of  a  surface. The   Teichm\"uller  space  is a  model  for $\eub{G}$. } 
\end{itemize}

Let  $G$ be a  discrete  group.  A  model for  the  classifying  space  for  free  actions $EG$  is a  free  contractible  $G$-CW  complex. Given a  model  $EG$ for  the  classifying  space  for  free actions, the space $BG$ is  the $CW$-complex $EG/G$.   

The  following  result  is  proved  in  \cite{joachimlueck},  Lemma  26  in page  6. 

\begin{lemma}
Let  $X$  be a  finite  proper  $G$-CW  complex. Then $X\times_G EG$  is  homotopy  equivalent  to  a  CW complex  of  finite  type. 
\end{lemma}
\subsection{Twisted  equivariant $K$-Theory.  }
Twisted  equivariant  K-Theory  for  proper  actions  of  discrete  groups  was  introduced  in \cite{barcenasespinozajoachimuribe}. 

In what follows we will recall its definition using Fredholm bundles and its properties following the  above mentioned article. The  crucial  difference  to  \cite{barcenasespinozajoachimuribe} is  the  use  of  graded  Fredholm  bundles, which  are  needed  for  the   definition  of  the multiplicative  structure.  

%\commentm{En el siguiente p\'arrafo se debe hablar del papel de la topolog\'ia compacto abierta. En particular que topolog\'ia vamos a usar. $*$-strong, compact open \'o compactly generated compact-open topology.}

Let $\HH$ be a separable Hilbert space and let 
$$\UU(\HH):= \{ U : \HH \to \HH \mid U\circ U^*= U^*\circ U = \mbox{Id} \}$$ the group
of unitary operators acting on $\HH$ with the compact open topology. Note that in $\UU(\HH)$ the compact open topology agree with the strong operator topology (Thm. 1.2 in \cite{EU}).

We consider this topology instead the norm topology because the last one is too restrictive for our purposes, for example for the regular representation $\HH=L^2(G)$ the action $$G\to \UU(\HH)$$ is not norm continuous. %Let ${\rm End}(\HH)$ denote the space of endomorphisms
%of the Hilbert space and endow ${\rm End}(\HH)_{c.o.}$ with the compact open topology. Consider the inclusion
%\begin{align*}
%\UU(\HH) &\to {\rm End}(\HH)_{c.o.} \times {\rm End}(\HH)_{c.o.}\\
%U &\mapsto (U,U^{-1})
%\end{align*}
Topologize the  group $P\UU(\HH)$  from the exact  sequence
$$1\to S^1\to  \UU(\HH)\to P\UU(\HH)\to  1 .$$

Let $G$ be a Lie group, a continuous homomorphism $a$ defined  on a  Lie  group $G$,  $a : G \to P\UU(\HH)$ is called  
stable if the unitary representation $\HH$ induced by the
homomorphism $\widetilde{a}: \widetilde{G}= a^*\UU(\HH) \to
\UU(\HH)$ contains each of the irreducible representations of
$\widetilde{G}$ infinitely often, where the  subgroup $S^1 \subset \widetilde{G}$ acts  by   scalar  multiplication.  

Here $\widetilde{G}$ and $\widetilde{a}$ denote respectively the
topological group and the continuous homomorphism defined by the pullback
square
$$\xymatrix{
\widetilde{G} \ar[r]^{\widetilde{a}} \ar[d] & \UU(\HH) \ar[d] \\
G \ar[r]^a & P\UU(\HH).}
$$

\begin{definition}\label{def projective unitary G-equivariant stable bundle}
Let  $X$ be  a proper $G$-CW  complex.  A projective unitary $G$-equivariant stable bundle over $X$
is a principal $P\UU(\HH)$-bundle
$$P\UU(\HH) \to P \to X$$
where $P\UU(\HH)$ acts on the right, endowed with a left $G$-action lifting the action on $X$ such that:
\begin{itemize}
\item the left $G$-action commutes with the right $P\UU(\HH)$
action, and \item for all $x \in X$ there exists a
$G$-neighborhood $V$ of $x$ and a $G_x$-contractible slice $U$ of
$x$ with $V$ equivariantly homeomorphic to $ U \times_{G_x} G$
with the action $$G_x \times (U \times G) \to U \times G, \ \ \ \
k \cdot(u,g)= (ku, g k^{-1}),$$ together with a local
trivialization
$$P|_V \cong  (P\UU(\HH) \times U) \times_{G_x} G$$ where the action of the isotropy group
is:
 \begin{eqnarray*}
 G_x \times \left( (PU(\HH) \times U) \times G \right)& \to & (P\UU(\HH) \times U) \times G
 \\
\ \left(k , ((F,y),g)\right)& \mapsto & ((f_x(k)F, ky), g k^{-1})
\end{eqnarray*} with $f_x : G_x \to PU(\HH)$ a fixed stable
homomorphism.
\end{itemize}

\end{definition}
%\commentm{Que topolog\'ia usamos?}
Notice that $P\UU(\HH)$ acts continuously on the projective space associated to $\HH$,  denoted by  $\IP(\HH)$. 

Given a  projective unitary $G$-equivariant stable bundle $P$ over $X$,  one  constructs  the  $G$-equivariant fiber bundle
$$\IP(\HH)\to P\times_{P\UU(\HH)}\IP(\HH)\to X$$
with structure group $P\UU(\HH)$.
\begin{definition}
	Let $P\times_{P\UU(\HH)}\IP(\HH)$ and $P'\times_{P\UU(\HH)}\IP(\HH)$ be two fiber bundles defined as above. The  fiber  bundle of Hilbert-Schmidt  homomorphisms, denoted by  $$L_{H-S}(P\times_{P\UU(\HH)}\IP(\HH),P'\times_{P\UU(\HH)}\IP(\HH))$$  is  the  $G$-equivariant  fiber  bundle  with structure group $P\UU(\HH)$, and   local  trivializations  given  as  follows.
	
	For all $x\in X$ there is a $G$-neighborhood $V$ of $x$ and a $G_x$-contractible slice $U$ of $X$ with $$V\cong_GU\times_{G_x}G$$
	and local trivializations 
    \begin{align*}P\times_{P\UU(\HH)}\IP(\HH)\mid V&\cong (\IP(\HH)\times U)\times_{G_x}G\text{, and}\\
    P'\times_{P\UU(\HH')}\IP(\HH)\mid V&\cong (\IP(\HH')\times U)\times_{G_x}G.	\end{align*}
The  bundle $L_{H-S}(P\times_{P\UU(\HH)}\IP(\HH),P'\times_{P\UU(\HH)}\IP(\HH))$  has  as  fiber  the  projective space  $\IP(L_{H-S}(\HH^*,\HH'))$ of  Hilbert-Schmidt   linear maps from  the  dual  of $\HH$ to  $\HH'$ in the  norm  topology. 

Thus  one  has a  local  trivialization  

$$L_{H-S}(P\times_{P\UU(\HH)}\IP(\HH),P'\times_{P\UU(\HH)}\IP(\HH))\mid V\cong\big(\IP(L_{H-S}(\HH^*,\HH'))\times U)\big)\times_{G_x}G.$$

As noted in \cite{atiyahsegal} on pages 5-6, although $\HH_x$ is not determined canonically by $\IP(\HH)$  and $\HH'$ neither by $\IP(\HH')$, the  bundle   $L_{H-S}(P\times_{P\UU(\HH)}\IP(\HH), P'\times_{P\UU(\HH)}\IP(\HH')) $     is  canonically  determined by $P$  and $P'$.

\end{definition}
\begin{definition}[Tensor product of projective unitary $G$-equivariant stable bundles]
	Let $P$ and $P'$ be projective unitary $G$-equivariant stable bundles over $X$. Define $P\otimes P'$ as the projective unitary $G$-equivariant stable bundle associated to the fiber bundle 
	$$L_{H-S}(P\times_{P\UU(\HH)}\IP(\HH),P'\times_{P\UU(\HH)}\IP(\HH)).$$
\end{definition}

In \cite{barcenasespinozajoachimuribe}, Theorem  3.8, the  set  of  isomorphism classes  of  projective  unitary stable $G$-equivariant  bundles, denoted  by  $Bun_{st}^G(X, P\UU(\HH))$   was  seen  to  be  in  bijection with  the set of  third  degree Borel  cohomology  classes  with  integer  coefficients $H^3(X\times_G EG, \IZ)$. This bijection can be extended to a group isomorphism.

\begin{proposition}\label{propcohomology}
The  map 
$$ Bun _{st}^G(X, P\UU(\HH)) \to H^3(X\times_G EG, \IZ) $$
is  an  abelian  group  isomorphism if  the  left  hand  side is  furnished  with  the  tensor  product  as  additive  structure. 

\end{proposition}
\begin{proof}
In Theorem 3.8 in \cite{barcenasespinozajoachimuribe} was constructed  a  classifying  $G$-space  $\mathcal{B}$, a  universal  projective  unitary stable $G$-equivariant bundle $\mathcal{E}\to \mathcal{B}$,  as  well  as  a  weak homotopy  equivalence 
$$f: Maps(X, \mathcal{B})^G \to Maps(X\times_G EG, B P\UU(\HH)).$$ 
 (This  was  only  stated  for  $\pi_0$   there,  but  the  argument works for higher  homotopy  groups). 
Moreover in Theorem 3.8 in \cite{barcenasespinozajoachimuribe} was constructed  a bijection  by taking the pullback of the universal bundle 
$$Bun_{st}^G(X, P\UU(\HH))\xrightarrow{\phi}\pi_0(Maps(X, \mathcal{B})^G).$$It is also proved that $Maps(EG,BP\UU(\HH))$ is a universal space for projective unitary stable $G$-equivariant bundles. From this fact it is clear that the tensor product operation corresponds with the Hopf space operation in $$Maps(EG,BP\UU(\HH))$$ induced by the operation in $BP\UU(\HH)$. Now using the group isomorphism
$$\pi_0(X,Maps(EG,BP\UU(\HH)))\cong \pi_0(X\times_GEG,BP\UU(\HH)),$$ we have that there is group isomorphism 

$$Bun_{st}^G(X, P\UU(\HH))\cong\pi_0(X\times_GEG,BP\UU(\HH))\cong  H^3(X\times_G EG, \IZ),$$where the last isomorphism is obtained because $BP\UU(\HH)$ is a $K(\IZ,3)$-space.

%On the other hand there is another bijection 
%$$\pi_0(Maps(X, \mathcal{B})^G)\to \pi_0(Maps(X\times_G EG, B P\UU(\HH))\big ).$$
%On the set $\pi_0(Maps(X, \mathcal{B})^G)$ define the operation $\star$ as follows:

%Given a pair of $G$-homotopy classes of maps $([f_0],[f_1])$ in $Maps(X, \mathcal{B})^G$, define $$[f_0]\star[f_1]=\phi^{-1}\left(f_0^*(\mathcal{E})\otimes f_1^*(\mathcal{E})\right).$$
%On the other hand on the set $\pi_0(Maps(X, \mathcal{B})^G)$ we define the operation $*$ as induced from the unique H-space structure in $BP\UU(\HH)=K(\IZ,3)$.
 
 % The  classification  of  bundles yields  that  these  operations  are  mutually  distributive and  associative, and  have a common  neutral  element  given  by the  constant  map. The  two  operations  agree  then because  of  the  standard lemma comparing homotopy  mutually  distribitive, $H$-space  structures on a  space.  See  for  example Lemma  2.10.10, page 56  in \cite{aguilargitlerprieto}.  

\end{proof}

\begin{definition}
Let  $X$  be a  proper  $G$-CW  complex and let $\HH$ be a separable Hilbert  space. The  space  $\Fred'(\HH)$
consists of pairs $(A,B)$ of bounded operators on $\HH$ such that
$AB -1$ and $BA -1$ are compact operators. Endow $\Fred'(\HH)$
with the topology induced by the embedding \begin{eqnarray*}
\Fred'(\HH) & \to & {\mathsf{B}}(\HH) \times  {\mathsf{B}}(\HH)
\times {\mathsf{K}}(\HH)
\times {\mathsf{K}}(\HH) \\
(A,B) & \mapsto & (A,B,AB-1, BA-1)
\end{eqnarray*}
where ${\mathsf{B}}(\HH)$ denotes the bounded operators on $\HH$ with
the compact open topology and ${\mathsf{K}}(\HH)$ denotes the compact
operators with the norm topology. There is a composition operation on $\Fred'(\HH)$ defined as
 \begin{align*}
 	\Fred'(\HH)\times\Fred'(\HH)&\xrightarrow{\circ}\Fred'(\HH)\\
 	\big((A,B),(A',B')\big)&\to (AA',B'B).
 	\end{align*}

\end{definition}

We denote by $\Hg=\HH\oplus \HH$  a $\IZ_2$-graded, infinite  dimensional  Hilbert space.

\begin{definition}

Let $\UU(\Hg)$ be  the  group  of  even, unitary  operators  on the  Hilbert  space $\Hg$ which  are  of  the  form 
$$\begin{pmatrix}  u& 0  \\ 0 & u      \end{pmatrix},$$

where $u$ denotes an element in $\UU(\HH)$. 

%in the  sense of \cite{barcenasespinozajoachimuribe}. This  is  
%the  topology determined  by  considering the  compactly generated  topology  associated  to the topology  on $U(\HH)$ given by  the  inclusion $U(\HH) \to {\rm End}(\HH)_{c.o.} \times {\rm End}(\HH)_{c.o.}$ to the space  of  endomorphisms of  the Hilbert space $\HH$, defined  as $g\mapsto (g, g^{-1})$  with  the  compact  open  topology. Notice that  this  topology is  different  to the  usual  compactly  generated  topology  in $U(\HH)$. 
We  denote  by  
$P\UU(\Hg)$ the  group $U(\Hg)/ S^1$ and recall that there is a  central  extension 
$$1\to  S^1\to \UU(\Hg)\to P\UU(\Hg)\to 1.$$ 

\end{definition}

\begin{definition} The space $\Fred''(\Hg)$ is  the  space  of pairs $(\widehat{A},\widehat{B})$ of self-adjoint, bounded operators of degree 1 defined on $\Hg$ such that $\widehat{A}\widehat{B}-I$ and $\widehat{B}\widehat{A}-I$ are compact. Endow $\Fred''(\Hg)$
with the topology induced by the embedding \begin{eqnarray*}
	\Fred''(\Hg) & \to & \BB(\Hg) \times  {\BB}(\Hg)
	\times {\KK}(\Hg)
	\times {\KK}(\Hg) \\
	(\widehat{A},\widehat{B}) & \mapsto & (\widehat{A},\widehat{B},\widehat{A}\widehat{B}-1, \widehat{B}\widehat{A}-1)
\end{eqnarray*}
where ${\BB}(\Hg)$ denotes the bounded operators on $\Hg$ with
the compact open topology and ${\KK}(\Hg)$ denotes the compact
operators with the norm topology.
\end{definition}

The space $\Fred''(\Hg)$ is homeomorphic to $\Fred'(H)$, we fix the homeomorphism $f$ sending
$$\begin{pmatrix}
0&A\\A^*&0
\end{pmatrix}\mapsto A.$$

\begin{definition}
 %Moreover one can deform $\Fred''(\Hg)$ as follows.
We denote by $\Fred^{(0)}(\Hg)$ the space of self-adjoint degree 1 Fredholm operators $\widehat{A}$ in $\Hg$ such that $\widehat{A}^2$ differs from the identity by a compact operator, with the topology coming from the embedding $\widehat{A}\mapsto (\widehat{A},\widehat{A}^2-I)$ in $\BB(\Hg)\times\KK(\Hg)$ (Here $\BB(\Hg)$ and $\KK(\Hg)$ have the compact open topology and the norm topology respectively). Note that there is a natural inclusion from $\Fred^{(0)}(\Hg)$ to $\Fred''(\Hg)$ defined as
\begin{align*}
	\Fred^{(0)}(\Hg)&\xrightarrow{i} \Fred''(\Hg)\\
\widehat{A}&\mapsto (\widehat{A},\widehat{A}).\end{align*}

\end{definition}

The  following  result  was  proved  in \cite{atiyahsegal}, Proposition 3.1 : 

\begin{proposition}
There is a deformation retract \begin{align*}\Fred''(\Hg)\xrightarrow{r}\Fred^{(0)}(\Hg).\end{align*}
\end{proposition}
%\begin{proof}
% Let $(\widehat{A},\widehat{B})$ be an element of $\Fred''(\widehat{H})$. Then we could write

% \begin{equation*}
%  \widehat{A}=
%    \begin{pmatrix}
%    0&A\\A^*&0
%    \end{pmatrix}
%  \text{ and }\widehat{B}=
%  \begin{pmatrix}
%  0&B\\B^*&0
%  \end{pmatrix}    
%  \end{equation*}
%  with $A$ and $B$ bounded operators on $\HH$. For every bounded operator $C$ we denote by $|C|$ the positive self-adjoint operator such that $|C|^2=C^*C$. With the above convention define
  
%  \begin{equation*}
%  \widetilde{A}=\begin{pmatrix}
%  0&|B|A\\A^*|B|&0
%  \end{pmatrix},
%  \end{equation*}
%it is self-adjoint degree 1 operator, but moreover $\widetilde{A}^2-I$ is compact, then $\widetilde{A}\in\Fred^{(0)}({\Hg})$. The map $(\widehat{A},\widehat{B})\mapsto \widetilde{A}$ is a deformation retract via the path 
%\begin{equation*}
%  \widetilde{A}_t=\begin{pmatrix}
%  0&|B|^tA\\A^*|B|^t&0
%  \end{pmatrix}.
%  \end{equation*}
%\end{proof}

The  above  discussion  can be  concluded by telling   that $\Fred^{(0)}(\Hg)$ is a representing space for $K$-theory. The group $\UU(\Hg)$ of degree 0 unitary operators on $\Hg$ with the compact open topology  acts continuously by conjugation on $\Fred^{(0)}(\Hg)$, therefore the same is true for the group $P\calu(\Hg)$. In \cite{barcenasespinozajoachimuribe} twisted $K$-theory for proper actions of discrete groups was defined using the representing space $\Fred '(\HH)$, but in order to have multiplicative structure we proceed using $\Fred^{(0)}(\Hg)$.

Let us choose the operator 
\begin{equation*}
  \widehat{I}=\begin{pmatrix}
  0&I\\I&0
  \end{pmatrix}.
  \end{equation*}

as the base point in $\Fred^{(0)}(\Hg)$.

Choosing  the identity  as  a  base point on the space  $\Fred^{'}(\mathcal{H})$,  we have a  diagram  of  pointed  maps

$$\xymatrix{\Fred^{(0)}(\Hg)\ar[r]^{i} &  \Fred^{''}(\Hg)\ar[d]^{r}\ar[r]^f & \Fred^{'}(\HH) \\ & \Fred^{(0)}(\Hg) &  } . $$
%\commentm{Hay que explicar mejor las funciones que se introducen el el siguiente p\'arrafo.}

 Moreover,  the  above maps  are  compatible  with the conjugation  actions  of  the  group $\UU(\Hg)$ on $\Fred''(\Hg)$, respectively of the group $\UU(\HH)$  on $\Fred'(\HH)$ , and with the  map \begin{align*} \,\UU(\Hg)\ \   &\to \UU(\HH)\\
 	\begin{pmatrix}  u& 0  \\ 0 & u      \end{pmatrix}&\mapsto \ \ u.\end{align*}

Let $X$ be a proper cocompact  $G$-CW-complex and  let $P \to X$  be a projective unitary stable
$G$-equivariant bundle over $X$. Denote by $\widehat{P}$  the  fiber  bundle  $$P\times_{P\UU(\HH)}\mathbb{P}(\Hg).$$

 The space $\Fred^{(0)}(\Hg)$  is endowed with a continuous right action (by homeomorphism)
of the group $P\calu({\calh})$ by conjugation, therefore we can take
the associated bundle over $X$
$$\Fred^{(0)}(\widehat{P}) := \widehat{P} \times_{P\calu(\calh)} \Fred^ {(0)}(\widehat{\calh}),$$
  and with the induced $G$-action given by
 $$g \cdot [(\lambda, A))] := [(g \lambda,A)]$$
for $g\in G$, $\lambda\in\widehat{P}$ and $A\in\Fred^{(0)}(\widehat{\calh})$.

Denote by $$\Gamma(X; \Fred^{(0)}(\widehat{P}))$$ the space of sections of the
bundle $\Fred^{(0)}(\widehat{P}) \to X$ and choose as base point in this space the
section which chooses the base point $\widehat{I}$ on the fibers. This section
exists because the ${P}\calu(\calh)$ action on ${\widehat{I}}$
is trivial, and therefore $$X \cong \widehat{P}/P\calu(\calh) \cong \widehat{P}
\times_{P\calu(\calh)} \{{\widehat{I}} \} \subset \Fred^{(0)}(\widehat{P});$$
let us denote this  section by $s$.

\begin{definition} \label{definition K-theory of X,P}
  Let $X$ be a connected $G$-space and $P$ a projective unitary stable
  $G$-equivariant bundle over $X$. The {\it{Twisted $G$-equivariant
  K-theory}} groups of $X$ twisted by $P$ are defined as the  homotopy  groups  of  the  $G$-equivariant  sections
  $$K^{-p}_G(X;P) := \pi_p \left( \Gamma(X;\Fred^{(0)}(\widehat{P}))^G, s \right)$$
  where the base point $s= \widehat{I}$ is the section previously constructed.

\end{definition}

\subsection{Topologies  on the  space of  Fredholm  Operators}

In \cite{tuxustacks} a Fredholm picture of twisted K-theory is introduced, using  the strong$^*$-operator  topology  on the  space  of  Fredholm Operators. For  the  sake  of  completeness,  we  establish  here  the  isomorphism  of  these  twisted  equivariant  $K$-theory groups with  the  ones  described  here. 

Denote by   $\Fred'(\HH)_{s*}$ the space whose elements are the same as $\Fred'(\HH)$ but with the strong$^\ast$-topology on $B(\HH)$.

\begin{definition}\cite[Thm. 3.15]{tuxustacks} \label{definition K-theory Tu-Xu}
  Let $X$ be a connected $G$-space and $P$ a projective unitary stable
  $G$-equivariant bundle over $X$. The {\it{Twisted $G$-equivariant
  K-theory}} groups of $X$ (in the sense of Tu-Xu-Laurent) twisted by $P$ are defined as the  homotopy  groups  of  the  $G$-equivariant  strong${}^\ast$-continuous sections
  $$\IK^{-p}_G(X;P) := \pi_p \left( \Gamma(X;\Fred'({P})_{s^*})^G, s \right).$$
  The bundle $\Fred'({P})_{s^*}$ is defined in a similar way as $\Fred'({P})$.

\end{definition}
We observe that the functors $K_G^\ast(-,P)$ and $\IK_G^\ast(-,P)$ are naturally equivalent.
\begin{lemma}
  The spaces $\Fred'(\HH)$ and $\Fred'(\HH)_{s^*}$ are weakly homotopy equivalent as $PU(\HH)$-spaces.
\end{lemma}
\begin{proof}
  The strategy is to prove that $\Fred'(\HH)_{s^*}$ is a representing space of equivariant K-theory. The same proof for $\Fred'(\HH)$ in \cite[Prop. A.2.2]{atiyahsegal} applies. In particular $U(\HH)_{s^\ast}$ is equivariantly contractible because the homotopy $h_t$ constructed in \cite[Prop. A.2.1]{atiyahsegal} is continuous in the strong$^\ast$-topology and then the proof applies.
\end{proof}

Using the above lemma one can prove that the pointset identity map from  $\Fred'(\HH)$ to $\Fred'(\HH)_{s^*}$ defines an equivalence between (twisted) cohomology theories $K_G^\ast(-,P)$ and $\IK_G^\ast(-,P)$. Then we have that both definitions of twisted K-theory are equivalent. Summarizing

\begin{theorem}
For every proper $G$-CW-complex $X$ and every projective unitary stable $G$-equivariant bundle over $X$ we have an isomorphism
$$K^{-p}_G(X;P)\cong \IK^{-p}_G(X;P).$$
\end{theorem}
%\begin{remark}
%In order to simplify the notation from now on we denote by $\HH$  a $\IZ_2$-graded separable Hilbert space and we denote by $\Fred^{(0)}({P})$   the bundle $\FredP$.
%\end{remark}

\subsection{Additive structure}
There  exists  a  natural  map  
$$\Gamma(X;\Fred^{(0)}(\widehat{P}))^G \times \Gamma(X;\Fred^{(0)}(\widehat{P}))^G \to \Gamma(X;\Fred^{(0)}(\widehat{P}))^G, $$ 
 inducing  an  abelian  group  structure  on  the twisted  equivariant  $K$- theory  groups,  which   we  will  define  below. Define $a$ as the composition of the following maps
 
 $$\Fred^{(0)}(\Hg) \times \Fred^{(0)}(\Hg) \xrightarrow{f\circ i }\Fred^{'}(\HH) \times \Fred^{'}(\HH)\xrightarrow{\circ}\Fred^{'}(\HH)\xrightarrow{r\circ f^{-1}} \Fred^{(0)}(\Hg)$$
 Note that in $\Fred'(\HH)$ the composition is well defined because the grading is not  present.
As  the  maps  involved  in the  diagram  are  compatible  with the  conjugation  actions  of  the   groups $\UU(\Hg)_{c.g}$, respectively $ \UU(\HH)_{c.g}$ and  $G$,  for  any  projective  unitary, stable  $G$-equivariant  bundle  $P$, this  induces  a pointed   map 

$$\Gamma(X;\Fred^{(0)}(\widehat{P}))^G, s) \times (\Gamma(X;\Fred^{(0)}(\widehat{P}))^G, s )\xrightarrow{\ \widetilde{a}\ }(\Gamma(X;\Fred^{(0)}(\widehat{P}))^G, s) .$$
Which  defines  an additive structure  in  $K^{-p}_G(X;P)$.

\subsection{Multiplicative structure}\label{sectionmultiplicative}
We  define an associative  product on twisted K-theory.

$$K^{-p}_G(X;P)\times K^{-q}_G(X;P')\rightarrow K^{-(p+q)}_G(X;P\otimes P')$$

induced by the map 
$$(A,A')\mapsto A\widehat{\otimes} I+I\widehat{\otimes} A'$$

defined in $\Fred^{0}({\Hg})$, and  $\widehat{\otimes}$ denotes  the  graded  tensor  product, see  \cite{atiyahsinger}  on pages  24-25 for  more details. We denote this product by $\bullet$.

Now we will describe how to endow to the graded group $K^*_G(X,P)$ with a $K^*_G(X)$-module structure and how to define the \emph{trivial} $G$-equivariant projective unitary bundle.
%\commentm{Hay que introducir notacion para el producto}
\begin{definition}[Def. 3.1 in \cite{refrito}]
	Let  $\alpha\in Z^ {2}(G,S^1)$  be a  normalized  torsion cocycle  of order  $n$   for  the  discrete  group  $G$,  with associated  central  extension  
	$$0\to  \mathbb{Z}/{n} \to  G_{\alpha}\to  G.$$  An  \textit{$\alpha$-twisted  vector  bundle}  is  a  finite  dimensional $G_{\alpha}$-equivariant  complex vector bundle  such  that  $\mathbb{Z}/n$  acts  by  multiplication  by a  primitive $n$-th root  of  unity.  The  \textit{$\alpha$-twisted, $G$-equivariant  K-theory}  groups $^{ \alpha }{K}^{0}_{G}( X)$ are  defined  as   the  Grothendieck  groups  of  the  isomorphism classes of  $\alpha$-twisted  vector  bundles over  $X$.
	
\end{definition}
Given  a proper  $G$-CW  complex $X$,  define $^{\alpha}K^{-n}_{G}(X)$  as  the  kernel  of  the  induced  map 

$$ ^{ \alpha }{K}^{0}_{G}( X\times S^{n})   \overset{{\rm incl}^{*}}{\to}       {  ^{ \alpha }{K}^{0}_{G}( X)}. $$

In Section 6.4 in \cite{barcenasespinozauribevelasquez} is described a method to assign to a normalized torsion cocycle $\alpha\in Z^2(G;S^1)$ a projective unitary, stable $G$-equivariant bundle $P_\alpha$ in such way that the groups ${}^\alpha K_G^*(X)$ are isomorphic to the groups $K_G^*(X;P_\alpha)$.

Let $\mathbf{0}$  be  the  projective  unitary, stable  $G$-equivariant  bundle  associated  to the trivial cocycle in $Z^2(G,S^1)$. By results in \cite{refrito} the  groups  $\pi_*\big (\Gamma(X;\Fred^{(0)}(\mathbf{0})\big )^G$ are canonically isomorphic to  \emph{untwisted}, equivariant,  representable $K$-Theory in  negative  degree   for  proper  actions. The  extended  version  via  Bott  periodicity are canonically isomorphic  with  usual \emph{untwisted}, equivariant  K-theory  groups for proper $G$-CW  complexes defined in \cite{lueckoliverbundles}.

 Now considering the external product defined above and the trivial twisting $\mathbf{0}$ we have a map
 
 $$K_G^{-p}(X;\mathbf{0})\times K^{-q}_G(X;P)\to K^{-(p+q)}_G(X;\mathbf{0}\otimes P). $$

But as we noted in the proof of Prop. \ref{propcohomology} tensor product of projective unitary, stable $G$-equivariant bundles corresponds when we take isomorphism classes to the addition in $H^3(X\times_GEG;\IZ)$ and $\mathbf{0}$ corresponds to the neutral element in $H^3(X\times_GEG;\IZ)$, then $P\otimes\mathbf{0}$ is isomorphic to $P$. Moreover by Prop. 3.7 in \cite{barcenascarrillovelasquez} the external product does not depends on the isomorphism classes of $P$ or $\mathbf{0}$, then we can conclude that we have a well defined module structure

$$K_G^{-p}(X)\times K_G^{-q}(X;P)\to K_G^{-(p+q)}(X;P).$$
 
 %\commentm{Aun creo que se debe decir algo mas sobre el iso entre $P$ y $\mathbf{0}\otimes P$}
 
\subsection{Bredon  Cohomology  and its \v{C}ech  Version}

(Untwisted) Bredon  cohomology  has  been  a useful  tool  to  approximate  equivariant  cohomology  theories  with  the  use  of  spectral  sequences  of  Atiyah-Hirzebruch  type  \cite{lueckdavis}, \cite{barcenasespinozauribevelasquez}. 

We will  recall a  version  of  Bredon  cohomology with  local  coefficients   which was  introduced  in  \cite{barcenasespinozauribevelasquez} and  compared there to  other  approaches.  
These  approaches  fit all  into  the  general  approach  of   spaces  over  a  category \cite{lueckdavis}, \cite{barcenasbrownrepresentability}. 

Let  $\calu= \{ U_\sigma \mid  \sigma \in  I \}$  be  an   open  cover  of  the  proper $G$-CW  complex  $X$ which  is  closed  under  intersections  and  has  the  property  that  each  open  set  $U_\sigma$   is   $G$-equivariantly homotopic  to  an  orbit $G/H_\sigma\subset U_\sigma$  for a  finite  subgroup  $H_\sigma$.
The  existence  of  such  a  cover,  sometimes  known  as  \emph  {contractible  slice  cover},   is guaranteed  for  proper  $G$-ANR's  by an appropriate  version  of  the  slice  Theorem (see \cite{antonyan}).

\begin{definition}

Denote  by $\cover$ the  category with  objects $\calu$  and where a  morphism  is  given  by an   inclusion $U_\sigma \to U_\tau$ . 
A  twisted  coefficient  system   with values  on  $R$-Modules   is  a  contravariant    functor  $\cover\to  R-{\rm Mod}$. 

\end{definition}

\begin{definition}
Let  $X$  be  a  proper  $G$-space with  a  contractible slice  cover  $\mathcal{U}$, and  let  $M$  be  a twisted  coefficient  system. Define  the Bredon  equivariant  homology  groups  with  respect  to $\mathcal{U}$  as  the  homology  groups   of  the category $\cover$ with coefficients in $M$,  

$$H^{n}_G(X,\mathcal{U}; M):= H^{n}(\cover, M).$$
These  are  the  homology  groups  of  the  chain  complex defined  as  the  $R$-module 
$$ C_{*}^\mathbb{Z}(\cover)\otimes_{\cover} M, $$ 

given  as  the  balanced tensor  product  of   the  contravariant, free  $\mathbb{Z}\cover$-chain  complex  $C_{*}^\mathbb{Z}(\cover)$ and  $M$. This  is  the  $R$-module  
$$\underset{U_\sigma \in \cover}\bigoplus R\otimes_R M(U_\sigma)/K $$
where  $K$  is  the  $R$-module  generated  by  elements  
$$ r\otimes x- r\otimes i^*(x), $$ 
for  an  inclusion  $i: U_\sigma\to  U_\tau$. 
\end{definition}

\begin{remark}[Coefficients of  twisted  equivariant  $K$-Theory on  contractible  covers] \label{remarkcoefficients}
 
 Let $i_\sigma: G/H_\sigma\to  U_\sigma \to X$ be  the inclusion of  a $G$-orbit into $X$ and  consider  the  Borel  cohomology  group  $H^3( EG\times _G G/H_\sigma, \IZ )$. Given a class $ [P] \in  H^3(EG \times_G X, \IZ)$, we have that $$i^*([P])\in H^3(EG\times_GU_\sigma,\IZ)\cong H^3(BH_\sigma,\IZ)\cong H^2(BH_\sigma,S^1).$$ Then to $i^*([P])$ we can associate a $S^1$-central extension of $H_\sigma$   $$1\to S^1 \to \widetilde{H_{P_\sigma}} \to  H_\sigma \to  1 $$ with $$K_G^*(U_\sigma,P)\cong{}^{i^*([P])}K_G^*(U_\sigma).$$
%associated  to  the    class  given  by  the  image of  $P$  under  the  maps 
%$$\omega_\sigma : H^3(EG\times X, \IZ)\overset{i_\sigma ^* }{\to} H^3(EG\times_G G/H_\sigma, \IZ )\overset{\cong}{\to} H^3(BH_\sigma, \IZ)\overset{\cong}{\to} H^2(BH_\sigma , S^1).  $$

Restricting  the   functors  $\ktheory{G}{0}{X}$ and $\ktheory{G}{1}{X}$ to  the  subsets $U_\sigma$ gives contravariant  functors  defined  on the  category $\cover$. 

In \cite{refrito} is proved that the groups ${}^{i^*([P])}K_G^*(U_\sigma)$ satisfy:

$${}^{i^*([P])}K_{G}^{j}(U_\sigma) = \begin{cases} R_{S^{1}}(\widetilde{H_{P_\sigma}})&  \text{if } j= 0   \\  0 & \text{if }  j= 1    \end{cases}$$

The  symbol $R_{S^{1}}(\widetilde{H_{P_\sigma}})$ denotes  the  subgroup  of  the  abelian  group  of  isomorphisms  classes of complex $\widetilde{H_{P_{\sigma}}}$-representations,  where  $S^1$  acts  by  complex  multiplication. 

\end{remark}
 
We  recall  the  key result  from  \cite{barcenasespinozauribevelasquez}, proposition  4.2 
\begin{proposition}\label{propositionspectralsequence}
The spectral sequence
associated to the locally finite and equivariantly contractible cover $\mathcal{U}$ and converging to  $\ktheory{G}{*}{X} $, has for second page $E_2^{p,q}$ the cohomology of $\cover$
with coefficients in the functor $\mathcal{K}^0_G( ?,P|_?)$ whenever $q$ is even, i.e.
\begin{equation}\label{spseq}
E_{2}^{p,q}:= 
H^{p}_G(X,\mathcal{U}; \mathcal{K}^0_G( ?,P|_{?})) 
\end{equation}
and is trivial if $q$ is odd.
Its higher differentials
$$d_{r}:E_{r}^{p, q}\to  E_{r}^{p+r, q-r+1}$$
vanish  for    $r$ even.
\end{proposition}

\section{Module  Structure  for  twisted  Equivariant  K-theory }\label{sectionmodule}

Let  $X$ be  a  proper  $G$-CW complex,  and let  $P$  be  a stable  projective  unitary  $G$-equivariant  bundle  over  $X$.   Recall  that  up  to $G$-equivariant  homotopy,  there  exists  a  unique  map  $\lambda: X\to \eub{G}$. The  map  $\lambda^*$  together  with the module structure defined in Section 1 endows to $K^{0}_G(X,P)$ with  a $K_G^0(\eub{G})$-module structure 
$$K^0_G(\eub{G})\times K_G^0(X,P)\xrightarrow{\lambda^*\times id} K_G^0(X)\times K_G^0(X,P)\xrightarrow{\bullet} K^{0}_G(X,P).$$
%which gives $K^{0}_G(X,P)$ the  structure  of  a  module  over  the  ring $K^0_G(\eub{G})$. 

We  will  analyze  the  structure  of    $K^0_G(\eub{G})$ as  a  ring. The  results  in  the  following  proposition are proved  inside  the  proofs of Theorem 4.3,  page  610 in \cite{lueckoliverbundles},  and Theorem 6.5, page 21 in \cite{lueckolivergamma}.  

\begin{proposition}\label{propositionnoetherianring}
Let  $G$  be  a  group  which  admits  a  finite  model    for  the  classifying space  for  proper  actions $\eub{G}$. Then,   
\begin{enumerate}
\item $K_G^0(\eub{G})$  is  isomorphic  to  the  Grothendieck group  of isomorphism classes $G$-equivariant,  finite  dimensional complex $G$-vector bundles. 
\item The  ring  $K^0_G(\eub{G})$ is  Noetherian
\item  Let  $\OrGF{G}{\mathcal{FIN}}$ be  the  orbit  category  consisting  of   homogeneous  spaces $G/H$  with  $H$ finite   and  $G$-equivariant  maps. Denote  by  $R(?)$  the  contravariant $\OrGF{G}{\mathcal{FIN}}$-module given  by  assigning  to an  object $G/H$ the  complex  representation  ring $R(H)$ and  to a  morphism $G/H\to G/K$  the restriction $R(K)\to  R(H)$. Then,  there  exists a ring homomorphism 
$$K_G^0(\eub{G})\to \lim_{\OrGF{G}{\mathcal{FIN}}}R(?)$$
which  has  nilpotent  kernel  and  cokernel. 

\item  Given a prime  number  $p$, there  exists a vector bundle $E$ of  dimension prime  to  $p$,  such  that  for  every point  $x\in \eub{G}$, the  character  of  the  $G_x$  representation $E\mid_x$  evaluated  at  an  element  of  order not  a  power of p is 0. 

%\item \commentm{Se debe considerar el caso p=0}
\end{enumerate}

\end{proposition}

\begin{proof}
\begin{enumerate}
\item This  is  proved in  \cite{lueckoliverbundles}, \cite{phillipsbook}, \cite{emersonmeyer}, 3.8 on pages 8-9. 
\item Given  a  finite  proper  $G$-CW  complex $X$, there  exists  an  equivariant  Atiyah-Hirzebruch  spectral  sequence converging  to  $K^*_G(X)$ with  $E_2$ term  given  by $E_2^{p,q}=\bredoncell{p}{K^q(G/?)}{X}$, where  the  right  hand  side  denotes \emph{untwisted} Bredon  cohomology, defined  over  the  Orbit  Category  ${\rm Or}_{\calfin}(G)$  rather than over  the   category  $\cover$.

% The group   $E_2^{p,q}$  can be  identified  with  Bredon  cohomology  with coefficients  in  the  representation  ring if  $q$  is  even and  is  zero  otherwise. 

Since  the  Bredon  cohomology  groups  in the $E_2$-term of  the  spectral sequence are finitely  generated  as abelian groups if  $\eub{G}$  is a   finite  $G$-CW  complex,  this proves that the limit is also finitely generated as abelian group then it is automatically Noetherian.

\item The  edge homomorphism  of  the Atiyah-Hirzebruch  spectral  sequence  of  \cite{lueckdavis}  gives a  ring  homomorphism  %(usar la descripcion del edge de \cite{lueckolivergamma} en seccion 6 y el lema 6.4 para el kernel, para el cokernel prop. 6.1)
$K_G^0(\underbar{E}G)\to \bredoncell{0}{R^?}{\underbar{E}G}$. The  right  hand  side  can be identified  with the  ring  $\lim_{\OrGF{G}{\mathcal{FIN}}}R(?)$ and the edge map can be described as the restriction to the $0$-skeleton of $X$ under the identification 

$$\bredoncell{p}{R^?}{X}=\im\left(K_G(X^{(1)})\to K_G(X^{(0)})\right).$$

Then kernel of edge homomorphism is just $$\ker\left(K_G(X)\xrightarrow{\res} K_G(X^{(0)})\right),$$but Lemma 6.4 in \cite{lueckolivergamma} proves that this ideal is nilpotent. 

On the other hand in Prop. 6.1 in \cite{lueckolivergamma} is proved that for $\xi\in H^0(X,R^?)$ there is a $k>0$ such that $\xi^k$ is in the image of the edge morphism, then the cokernel is nilpotent.% The  rational  collapse  of  the  equivariant Atiyah-Hirzebruch spectral  sequence  gives  the  second  part.   

\item %Let  $m$  be  the  least  common multiple of  the  orders  of  isotropy  groups $H$ in $\eub{G}$. For  any  finite  subgroup  $H$, pick up  a  homomorphism $\alpha_H:H\to \Sigma_m$  corresponding  to a  free  action  of  $H$ on $\{1,\ldots m\}$.  Let  $n$  be  the   order  of  the  group $\Sigma_m/Syl_p(\Sigma_m)$ and  let $\rho:\Sigma_m\to U(n) $ be  the  permutation  representation.  Consider  the  element $\{V_H \}= \{\mathbb{C}^n[\rho\circ \alpha_{H}]\}$  in the  inverse  limit  $  \lim_{\OrGF{G}{\mathcal{FIN}}}R(?)$. According  to  the  second  part, there  exists  a  vector bundle  $E$  which  is  mapped  to  some  power $\{{V_H}^{{\otimes}^ k} \}$. The  vector  bundle  satisfies  the  required  properties.  
 It is just Corollary 2.9 in \cite{lueckoliverbundles}
\end{enumerate}

\end{proof}

\begin{lemma}\label{propositionnoetherianmodule}

 Let $G$ be a  discrete  group admiting  a  finite  model  for  $\eub{G}$  and  $P$  be  a  stable projective unitary  $G$-bundle over  a  finite $G$-CW complex $X$. Then,  the $K_G^0(\eub{G})$-modules $K_G^i(X, P)$ are  noetherian for  $i=0,1$.

\end{lemma}

\begin{proof}
It is proved in Thm. 5.3 in \cite{barcenasespinozauribevelasquez} that there is a  spectral  sequence converging  to $K_G^*(X,P)$. Its  $E_2$  term  consists  of  groups which  can  be  identified with  a version of  Bredon  cohomology  associated  to  an open, $G$-invariant cover $\mathcal{U}$  consisting  of open sets  wich  are  $G$-homotopy  equivalent  to proper  orbits.  

These  groups  are  denoted  by $\bredoncover{p}{K^q_G(\mathcal{U}}{X})$  and  are  zero   if  q is odd. Since $X$  is a  proper, $G$-compact, $G$-CW complex, the  cover can  be  assumed  to be  finite.  Given an  element  of  the  cover $U$, the group $K_G^{0}(U)$ is a finitely  generated,  free  abelian group, as  it  is  seen  from  A.3.4, page  40 in  \cite{barcenasespinozajoachimuribe},  where  the  groups $K_G^{0}(U)$ are identified  with  groups of projective  complex  representations. Compare  also  remark \ref{remarkcoefficients}. 

In particular the  groups $\bredoncover{p}{K^q_G(\mathcal{U}}{X})$  in the   spectral  sequence  converging  to $K_G^*(X,P)$ are  finitely  generated. By  induction,  the groups  $E_r^{p,q}$ are  finitely  generated  for  all  $r$  and  hence  the  term  $E_\infty$. Hence $K_G^i(X,P)$ is it  for  $i=0,1$. Since  $K_G^0(\eub{G})$  is  a noetherian  ring,  the  result follows. 
\end{proof}

\section{The  completion Theorem} \label{sectionatiyahsegal}

\begin{definition}[Augmentation ideal]
Let  $G$  be  a  discrete  group. Given  a proper  $G$-CW  complex, the augmentation ideal the augmentation ideal ${\bf I}_{G,X}\subseteq K_G(X)$ is defined to be the set of elements represented by virtual $G$-vector bundles of dimension zero on all connected components. In other words,
$${\bf I}_{G,X}=\ker\left(K_G(X)\xrightarrow{\dim}\prod_{\pi_0(X)/G}\IZ\right)$$

%${\bf I}_{G,X}\subset K_G^0(X)$ is  defined  to be  the  kernel  of  the  homomorphism 
%$$K_0^G(X)\to K_G^0(X_0)\to K_{\{e\}}^0(X_0)$$
%defined  by  restricting  to  the  zeroth skeleton  and  restricting the  acting group  to  the  trivial  group. 
\end{definition}

\begin{proposition}\label{propositionpowers}
Let  $X$  be  an n-dimensional  proper  $G$-CW complex. Then, any product  of n+1  elements in $\idealGX{G}{X}$  is  zero.  
\end{proposition}
\begin{proof}
This  is  proved  in Lemma 4.2 in   \cite{lueckoliverbundles}. 
\end{proof}

We  fix  now  our  notations  concerning  pro-modules  and  pro-homomorphisms.

Let  $R$ be  a ring. A  pro-module  indexed  by  the  integers  is  an inverse
 system of $R$-modules.
$$ M_{0}\overset{\alpha_1}{\leftarrow}  M_{1}\overset{\alpha_2}{\leftarrow}
 M_{2}\overset{\alpha_3}{\leftarrow} M_{3}, \ldots$$

 We  write $\alpha_{n}^{m}= \alpha_{m+1}\circ\ldots \circ \alpha_{n}:
 M_{n}\to M_{m}$  for  $n>m$  and  put $\alpha_{n}^ {n}={ \rm
 id}_{M_{n}}$.

 A \emph{strict}  pro-homomorphism $\{ M_{n},
 \alpha_{n}\}\to \{ N_{n}, \beta_{n}\}$  consists  of a  collection of
 homomorphisms $\{f_{n}: M_{n}\to N_{n}\} $  such that $\beta_{n}\circ
 f_{n}= f_{n-1}\circ\alpha_{n}$ holds  for  each $n\geq 2$. 
A  pro $R$-module $\{M_{n}, \alpha_{n}\}$ is  called  pro-trivial  if
 for  each $m\geq1$ there  is  some  $n\geq m$ such that $\alpha_{n}^
 {m}=0$. A  strict  homomorphism $f$ as  above  is called a pro
 isomorphism if  $\ker(f)$ and $\rm{coker}(f)$ are  both pro-trivial. A
 sequence  of  strict  homomorphisms  
$$\{M_{n}, \alpha_{n}\} \overset{\{f_{n}\}}{\to} \{M^ {'}_{n}, \alpha^
 {'}_{n}\} \overset{\{g_{n}\}} {\to} \{M^ {''}_{n}, \alpha^ {''}_{n}\}
 $$
 is  called pro-exact if $g_{n}\circ f_{n}=0$  holds for $n\geq 1$ and
 the pro-R-module $\{ \ker(g_{n}) /  { \rm im}(f_{n}) \}$ is
 pro-trivial. The  following  lemmas  are proved in
 \cite{atiyahmacdonald}, Chapter 10, section 2, see  also \cite{lueckoliverbundles}:  

\begin{lemma}\label{lemmalim}
Let $0\to \{ M^ {'}, \alpha^{'}_{n} \} \to \{ M_{n} , \alpha_{n} \} \to 
\{ M^ {''}_{n}, \alpha^{''}_{n}\}\to 0$ be a pro-exact  sequence  of
pro-$R$-modules.  Then,  there  is  a natural exact sequence 

\begin{multline*}
$$
 0\to {\rm invlim} M^{'}_{n} \overset{{\rm
    invlim}f_{n}}{\longrightarrow} {\rm invlim} M_{n} \overset{{\rm
    invlim}g_{n}}{\longrightarrow} {\rm invlim}
     M^{''}_{n}\overset{\delta}{\to} \\ {\rm invlim}^ {1} M^{'}_{n}
    \overset{{\rm    invlim}^ {1}f_{n}}{\longrightarrow} {\rm invlim}^
    {1} M_{n}  \overset{{\rm invlim}^ {1}g_{n}}{\longrightarrow} {\rm
    invlim}^ {1}     M^{''}_{n}
$$
\end{multline*} 
In particular, a pro-isomorphism $\{f_{n}\}: \{M_{n}, \alpha_{n}\}\to
\{ N_{n}, \beta_{n}\}$ induces  isomorphisms

$$
 {\rm invlim}_{n\geq 1} f_{n}: {\rm invlim}_{n\geq 1}M_n
  \overset{\cong}{\to} {\rm invlim}_{n\geq 1} N_{n}$$$$ {\rm
  invlim}^{1}_{n\geq 1} f_{n}: {\rm invlim}^ {1}_{n\geq 1}M_n
  \overset{\cong}{\to} {\rm invlim}^ {1}_{n\geq 1} N_{n}
$$    
%\end{multline*}  

\end{lemma}

\begin{lemma}\label{lemmaexact}
Fix  any commutative noetherian ring  $R$ and any  ideal $I\subset
R$. Then, for  any exact sequence $M^ {'} \to M \to M^ {''}$ of
finitely  generated $R$-modules, the  sequence 
$$ \{ M^ {'}/I^{n} M^ {'}\} \to \{ M/I^{n} M \} \to \{ M^ {''}/I^{n} M^
{''}\}$$
of pro-$R$-modules is pro-exact.   
\end{lemma}
\begin{proof}See Lemma 4.1 in \cite{lueckoliverbundles}.
	\end{proof}

%\commentm{explicar un poco mas la construccion de esta funcion}
\begin{definition}[Completion Map]
Let  $X$ be  a  proper $G$-CW  complex and let $P$ be a projective unitary $G$-equivariant stable bundle. Let $$p:X\times EG\to X$$ be  the  projection  to the first  coordinate. 
 Let $$\lambda: X\to  \eub{G}$$ be the unique continuous $G$-map up to $G$-homotopy, %combined  with Proposition \ref{propositionpowers} defines a  pro-homomorphism 
 define $\bar{\varphi}_{\lambda,p}$ as the following composition
 
 \begin{align*}K_G^*(X,P)&\xrightarrow{p^*}K_G^*(X\times EG,p^*(P))\xrightarrow{q}K(X\times_GEG,q(p^*(P)))\\&\xrightarrow{\res}K^*((X\times_GEG)^{(n-1)},\res(q(p^*(P)))).\end{align*}
 Where $q$ is the quotient map and $\res$ denotes the restriction. For simplicity we denote $\res(q(p^*(P)))$ by $P_{n-1}$.
 
 Via the map $\lambda$ we have that $K_G(X,P)$ is a $K_G(\eub{G})$-module. Note that the restriction
 
 $${\bf I}^n_{G,\eub{G}}\cdot K_G(X,P)\subseteq K_G(X,P)\xrightarrow{\bar{\varphi}_{\lambda,p}} K((X\times_G EG)^{(n-1)},P_{n-1})$$
 is zero, since  its image is contained in ${\bf I}^n_{\{0\},(EG\times_GX)^{(n-1)}}\cdot K((X\times_G EG)^{(n-1)},P_{n-1})$ and this ideal is zero by Lemma 4.2 in \cite{lueckoliverbundles}.
 We define the Completion Map as the pro-homomorphism induced by $\bar{\varphi}_{\lambda,p}$,
 $$\varphi_{\lambda,p}: \bigg \{ K_G^*(X,P)/ {{\bf I}^n_{G,\eub{G}}} \cdot K_G^*(X,P)  \bigg \}\longrightarrow  \bigg \{K^*((X\times_GEG)^{(n-1)},P_{n-1})\bigg \} $$
\end{definition}

\begin{theorem}\label{theoremcompletion}
Let  $G$  be a  group  which  admits  a  finite  model  for $\eub{G}$, the  classifying  space  for  proper  actions. Let  $X$ be  a  finite, proper $G$-CW  complex and  let $P$ be a  projective  unitary stable  $G$-equivariant  bundle  over $X$. Then, the  pro-homomorphism 
$$\varphi_{\lambda,p}: \big \{ K_G^*(X,P)/ {\bf I}^n_{G,\eub{G}}\cdot K_G^*(X,P)  \big \}\longrightarrow  \big \{K^*((X\times_G EG)^{(n-1)},P_{n-1})\big \} $$
is  a  pro-isomorphism. In  particular, the  system $\big \{K^*((X\times_G EG)^{(n-1)},P_{n-1})\big \}$ satisfies  the  Mittag-Leffler  condition  and  the  $\lim^ 1$ term is  zero. 
\end{theorem}

\begin{proof}
Due  to  propositions  \ref{propositionnoetherianring} and \ref{propositionnoetherianmodule},  we  are  dealing  with  a  noetherian  ring $K_G^0(\eub{G})$  and  the  noetherian  modules $K_G^*(X,P)$ over it. Hence,  we  can use  lemmas \ref{lemmaexact} and \ref{lemmalim}, and  the  5-lemma  for  pro-modules  and  pro-homomorphisms to  prove the  result  by  induction  on the  dimension  of $X$ and  the  number  of  cells  in each  dimension. 

Assume  that $X= G/H$  for  a  finite  group $H$. Then, the  completion  map fits  in the  following  diagram
$$\xymatrix{ \bigg \{ K_G^*(G/H,P) / \idealGX{G}{\eub{G}}^n \bigg \} \ar[r]^{\varphi_{\lambda,p}} &   \bigg\{ K_G^*(G/H\times EG^{(n-1)}, p^*(P))\bigg \} \\ 
  \bigg \{ K_H^*(\pt, P \mid_{eH})/ J^n \bigg \} \ar[u]^{{\rm ind}_{H\to G}}_{\cong}\ar[d]&  \bigg \{ K_H^*(EH^{(n-1)}, p^*(P)\mid_{eH\times EG^{(n-1)}})\bigg \} \ar[u]^{{\rm ind}_{H\to G}}_{\cong}\ar[d]^{=} \\   \bigg \{ K_H^*(\pt, P \mid_{eH})/ \idealGX{H}{\pt}^n\bigg\} \ar[r] & \bigg \{ K_H^*(EH^{(n-1)}, p^*(P)\mid_{eH\times EG^{(n-1)}})\bigg \}  }.  $$ 

The higher  vertical maps  are induction  isomorphisms (see \cite[Sec. 4.3]{barcenasespinozajoachimuribe}), and   the   ideal $J$ is  generated  by  the  image  of $\idealGX{G}{\eub{G}}$  under  the  map  $(\ind_{H\to G})^{-1}\circ \lambda$.  The  lower horizontal  map  is  a pro-isomorphism  as a  consequence  of  the Atiyah-Segal  Completion Theorem  for  Twisted  Equivariant  $K$-theory  of finite groups, Theorem 1  in \cite{lahtinenatiyahsegal}, where  it  is  proved  even for  compact  Lie  groups. We  will  analyze  now  the  lower  vertical  map and  verify  that  it  is  a  pro-isomorphism of pro-modules. 
This amounts to prove  that $\idealGX{H}{\pt}/ J$ is  nilpotent. Since the  representation  ring  of  $H$, $R(H)$ is  noetherian,  this holds  if  every prime  ideal  which  contains $J$  also  contains $\idealGX{H}{\pt}$.
For  an  element $v\in R(H)$,  denote by  $\chi_v$  the character of $v$.
Let  $H$  be  a  finite  group. Let  $\zeta=e^{\frac{2\pi i}{|H|}}$.  Put  $A=\IZ[\zeta]$.

Recall   Lemma  6.4 in \cite{atiyahcharactersfinitegroups},  that  given a finite  group $H$, every  prime  ideal $\mathcal{P}$ of  the  representation  ring $R(H)$ is the restriction of a prime ideal in $R(H)\otimes A$, moreover any prime ideal in $R(H)\otimes A$ is the restriction of a prime ideal in the ring of all $A$-valued maps on $G$, denoted by $A^H$.

The prime ideals in $A^H$ are described as follows, if $s\in H$ and ${\bf p}$ is a prime ideal in $A$, then $$\{\psi\in A^H\mid \psi(s)\in{\bf p}\}$$
is a prime ideal of $A^H$ and every prime ideal of $A^H$ is of this form, then for any prime ideal $\mathcal{P}\subseteq R(H)$ there is a prime ideal ${\bf p}$ in $A$ and $s\in H$ such that$$\mathcal{P}=\{v\in R(H)\mid \chi_v(s)\in{\bf p} \}$$

 Let $\mathcal{P}$ be  a   prime  ideal  not containing ${\bf I}_{H,\pt}$, then $s\neq e$, set $p=\charac(A/{\bf p})$ (possibly $p=0$), by arguments contained in \cite{atiyahcharactersfinitegroups} we can suppose that $s$ has order prime to $p$.% We  can  assume  that  there  exist  $s, t\in H$    with $\chi^{-1}_{s}(t) \in {\bf p}$ and  such  that if   $p$  is  the  characteristic  of  the  field $A/{\bf p}$, then the  order  of  $s$ is  prime to  $p$. 
 
 Let $\lambda:G/H\to \eub{G}$ be the unique $G$-map up to $G$-homotopy, let $x=\lambda(eH)$, then $G_x\supseteq H$, then every element  $\epsilon\in R(G_x)$ induces and element $\res(\epsilon)\in R(H)$.
 
 According  to part  (iv) of  Proposition \ref{propositionnoetherianring}, there  exists  a complex  $G$-vector  bundle $E$ over  $\eub{G}$  such  that $p$ is prime to ${\rm dim}_\mathbb{C} E$, and  the  character $\chi_{E\mid_{x}}(s) $. Let  $k={\rm dim}_\IC E$, and  $v=[\mathbb{C}^k]- \res([E \mid_{x}])\in R(H)$ then $[\IC^k\times\eub{G}]-[E]\in {\bf I}_{G,\eub{G}}$, then $v\in J$.  On the other hand, $\chi_v(s)=k$, but as $p$ is prime to $k$ it implies that $k\notin{\bf p}$, then $\mathcal{P}$ not contains $J$ and then ${\bf I}_{H,\pt}/J$ is nilpotent.

This  proves  that  the lower  horizontal  arrow  is  a  pro-isomorphism, the $\lim^1$ term is  zero, and the theorem  holds  for  0-dimensional  $G$-CW  complexes  $X$.
Assume  that  the  theorem  holds  for  all $n-1$-dimensional, finite proper $G$-CW  complexes. 
 Given a $k$-dimensional, finite, proper $G$-CW  complex, $X$ there exists  a  pushout 
 $$\xymatrix{\coprod_\alpha S^{k-1}\times G/H \ar[d] \ar[r]& \coprod_\alpha D^k\times G/H \ar[d] \\  Y \ar[r]& X}$$
 where  $Y$  is  a  $k-1$-dimensional, finite  proper  $G$-CW  complex. The  Mayer-Vietoris  sequence  for  twisted equivariant $K$-theory  gives   pro-homomorphisms

\begin{multline*}\ldots \bigg \{ K_G^*(X,P)/ {{\bf I}_{G,\eub{G}}}^n  \bigg \}   \longrightarrow \\  \bigg \{ K_G^*(Y,P)/ {{\bf I}_{G,\eub{G}}}^n   \bigg  \} \bigoplus \bigoplus_{\alpha} \bigg \{ K_G^*(D^k\times G/H ,P)/ {{\bf I}_{G,\eub{G}}}^n   \bigg \}  \longrightarrow \\  \bigoplus_{\alpha} \bigg \{ K_G^*(S^{k-1}\times G/H ,P)/ {{\bf I}_{G,\eub{G}}}^n   \bigg \} \longrightarrow \{ K_G^{*+1}(X,P)/ {{\bf I}_{G,\eub{G}}}^n  \bigg \} \ldots
\end{multline*}

By  induction,  the  completion  maps  for the  $n-1$-dimensional  $G$-CW  complexes  are  isomorphisms. By  the  5-lemma  for  pro-groups,  the   completion  map  for  $X$  is an  isomorphism.

\end{proof}
%\commentm{explicar un poco mas el siguiente corolario.}
\begin{corollary}
Let  $G$  be  a  discrete  group  with a  finite  model  for  $\eub{G}$. Let $P$ be a projective unitary stable $G$-equivariant bundle over $\eub{G}$, with $[P]\in H^3(\eub{G} \times_G EG, \mathbb{Z})\cong H^3(BG, \mathbb{Z})$. Consider  $I={\bf I}_{G,\eub{G}}$ Then there is a pro-isomorphism, 
$$\varphi_{id,p}: \big \{ K_G^*(\eub{G},P)/ {\bf I}^n_{G,\eub{G}}\cdot K_G^*(\eub{G},P)  \big \}\longrightarrow  \big \{K^*(BG)^{(n-1)},P_{n-1})\big \}. $$

\end{corollary}

\begin{proof}
	
	Using the completion theorem with $X=\eub{G}$ we have that 
	
	$$\varphi_{id,p}: \big \{ K_G^*(\eub{G},P)/ {\bf I}^n_{G,\eub{G}}\cdot K_G^*(\eub{G},P)  \big \}\longrightarrow  \big \{K^*(\eub{G}\times_GEG)^{(n-1)},P_{n-1})\big \} $$ is a pro-isomorphism, but as $X\times_GEG$ is a model for $BG$, then the required map is a pro-isomorphism and result follows.
	\end{proof}

\section{The  cocompletion Theorem }\label{sectioncocompletion}

We  will  now  dualize  the  results concerning  the   completion  theorem  in order to  obtain a  result  relating twisted  K-homology  of  the  Borel  construction $X\times_G EG$  of  a  finite proper $G$-CW  complex $X$  to the K-theory  ring  of  the  classifying  space  of  proper  actions   $K_{G}^{*}(\eub{G})$.

Given  a CW  complex  $X$, and  a  projective unitary stable $G$-equivariant bundle $P$, the  twisted  K-homology  groups are  defined  in terms of  Kasparov  bivariant  groups involving   continuous trace  algebras. See  the definition  below. 

The  comparison of  different  versions  of  twisted,  (nonequivariant!) $K$- homology, including  the  analytical one, involving   continuous  trace  algebras  has  been  
performed  in \cite{baumcareywang}.

A  generalized  construction for  the  Equivariant  version  of Twisted  $K$-homology  has  been  done  in \cite{barcenascarrillovelasquez}. The  comparison  to   other  methods including  continuous  trace  algebras, as  well  as crossed  products  and  KK-Theory  has been done  using  consequences  of  analytical  Poincar\'e Duality  in \cite{barcenasgeometrickhomology}.

We  refer  the  reader  for  preliminaries  on Kasparov  KK-Theory and  its  relation  to $K$-homology  and  Brown-Douglas-Fillmore Theory of  extensions to \cite{blackadar}, Chapter VII.

Recall  that he  norm  topology  and  the  compactly  generated  topology  agree on  compact  operators. Hence,   there  is  a  conjugation  action  of  the  group $\UU(\HH)$  of  unitary  operators  in the compact open  topology.

\begin{definition}[Continuous  trace Algebras]
Let  $X$  be  a  CW  complex. Given a  principal projective  unitary  bundle $P:E\to X$  , the  continuous  trace  algebra associated  to  $P$  is  the  algebra $A_P$ of  continuous  sections  of  the  bundle 
$$\mathcal{K}\times_{P\UU(\HH)} E\to X.$$ 
\end{definition}

\begin{definition}[KK-picture of  twisted K-homology]
Let  $X$  be a  locally  compact  space and  $P$  be  a  $P\UU(\HH)$-principal bundle. The  twisted  equivariant  $K$-homology  groups associated  to the  projective  unitary   principal bundle  $P$   are  defined   as  the  KK-groups 
$$K_*(X, P)= KK_*( A_P, \mathbb{C})$$  
\end{definition}

Continuous  trace  algebras,  used  in the  operator  theoretical  definition  of  twisted $K$-theory  and  $K$-homology  belong   to the  Bootstrap  class  \cite{blackadarvonneumann} Proposition IV.1.4.16. Hence,  the  following  form  of the  Universal  Coefficient Theorem for  $KK$-Groups holds. It  was  proved  in  \cite{rosenbergschochet}, Theorem 1.17:

\begin{theorem}[Universal  coefficient  Theorem  for  Kasparov KK-Theory]\label{theoremrosenbergschochet}
Let  $A, B$  be  $C^*$-algebras in the bootstrap category. Then, there  is  an  exact  sequence 
\begin{align*}
0\to \Ext_\mathbb{Z}(KK_{*-1}(\IC,A),KK_{*}(\IC,B)) &\to KK_*(A, B) \to\\ &\Hom_\mathbb{Z}(KK_*(\IC,A), KK_*(\IC,B))\to 0.  \end{align*}

\end{theorem}

Specializing  to  the  algebras $A_P$  one  has: 

\begin{theorem}\label{theoremuct}

Let  $X$  be a  locally  compact  space and  $P$  be  a  $P\UU(\HH)$-principal bundle. Then,  there  is  an  exact  sequence 

$$0\to \Ext_\mathbb{Z}(K^{*-1}(X, P), \mathbb{Z})\to K_*(X,P)\to \Hom_\mathbb{Z}(K^{*}(X,P), \mathbb{Z})\to 0 $$

\end{theorem}
 
 We  will  prove  the  following  cocompletion  Theorem,  inspired  by  the   methods  and  results of  \cite{joachimlueck}.

\begin{theorem}\label{theoremcocompletion}
Let $G$  be  a  discrete  group. Assume  that  $G$  admits  a  finite  model  for  $\eub{G}$. Let $X$ be a   finite  $G$-CW complex  and   $P$ be a projective unitary stable $G$-equivariant bundle. Let  $\idealGX{G}{\eub{G}}$ be  the  augmentation ideal. Then,  there  exists a  short  exact  sequence  
\begin{multline*}0\to\colim_{n\geq1} \Ext^{1}_\mathbb{Z} (K^*_G(X,P)/\idealGX{G}{\eub{G}}^n\cdot K_G^*(\eub{G},P) , \mathbb{Z})   \to \\ K_*(X\times_G EG, p^*(P)) \to \colim_{n\geq 1} \Hom(K^*_G(X,P)/\idealGX{G}{\eub{G}}^n\cdot K_G^*(\eub{G},P) ,\IZ)\to0
\end{multline*}
\end{theorem}
\begin{proof}
Choose a  CW  complex $Y$  of  finite  type and  a  cellular  homotopy  equivalence $f:Y \to  X\times_G EG$. Let  $f^{(n)}: Y^{(n)}\to  (X\times_G EG)^{(n)}$  be  the  map  restricted  to  the  skeletons. The   pro-homomorphisms 
$$\bigg \{K^*((X\times_G EG)^{(n)}, f^*(p^*(P)))\bigg \} \longrightarrow \bigg \{ K^*(Y^{(n)}, p^*(P)\mid Y^{(n)})\bigg \}$$
are a pro-isomorphism  of  abelian  pro-groups because it is induced by a $G$-homotopy equivalence.
On the  other hand,  due  to  the  completion  theorem,  \ref{theoremcompletion},  there  are  pro-isomorphisms

$$\varphi_{\lambda,p}: \bigg \{ K_G^*(X,P)/ {{\bf I}^n_{G,\eub{G}}}\cdot K_G^*(X,P)  \bigg \}\longrightarrow  \bigg \{K^*((X\times_G EG)^{(n-1)},P_{n-1})\bigg \} $$

Using  \ref{theoremuct},  one  gets  the  exact  sequence  

\begin{align*}0\to \Ext_\mathbb{Z} &( K_{*-1}(Y, f^*(p^*(P))), \mathbb{Z}) \to\\& K^*(Y, f^*(p^*(P))) \to \Hom_\mathbb{Z}(K_*(Y, f^*(p^*(P))), \mathbb{Z})\to 0.\end{align*}  

Combining this  exact  sequence with the  pro-isomorphisms  given  previously,  one  has  the  exact  sequence (it is because $\colim_{n\geq1}$ is an exact functor)
\begin{multline*}0\to\colim_{n\geq1} \Ext^{1}_\mathbb{Z} (K^*_G(X,P)/\idealGX{G}{\eub{G}}^n\cdot K_G^*(\eub{G},P) , \mathbb{Z})   \to  K_*(X\times_G EG, p^*(P)) \to \\\colim_{n\geq 1}\Hom_\IZ( K^*_G(X,P)/\idealGX{G}{\eub{G}}^n\cdot K_G^*(\eub{G},P) ,\IZ) \to0
\end{multline*}

\end{proof}

\bibliographystyle{abbrv}
\bibliography{atiyahsegal}

\end{document}